\documentclass{article}
\usepackage{graphicx}
\usepackage{times}
\usepackage{amsthm}
\usepackage{amsmath}
\usepackage{amssymb}
\usepackage{mathrsfs}
\usepackage{pb-diagram}
\usepackage{color}
\usepackage[colorlinks]{hyperref}
\renewcommand{\o}{\omega}

\renewcommand{\th}{\theta}

\renewcommand{\O}{\Omega}
\newcommand{\N}{\mathbb N}
\newcommand{\T}{\mathbb T}

\newcommand{\X}{\mathcal X}
\newcommand{\Y}{\mathcal Y}
\newcommand{\Z}{\mathcal Z}

\renewcommand{\H}{\mathcal H}

\newcommand{\R}{\mathbb R}

\newcommand{\W}{\mathscr W}

\renewcommand{\S}{\mathcal S}
\newcommand{\e}{\sf e}
\newcommand{\g}{\mathfrak g}
\newcommand{\im}{\mathop{\mathsf{Im}}\nolimits}
\newcommand{\re}{\mathop{\mathsf{Re}}\nolimits}

\renewcommand{\[}{\left[}

\renewcommand{\]}{\right]}

\newcommand{\G}{{\sf G}}
\newcommand{\wG}{\widehat{\sf{G}}}
\newcommand{\p}{\parallel}
\newcommand{\<}{\langle}
\renewcommand{\>}{\rangle}

\newtheorem{Theorem}{Theorem}[section]
\newtheorem{Remark}[Theorem]{Remark}
\newtheorem{Lemma}[Theorem]{Lemma}

\newtheorem{Corollary}[Theorem]{Corollary}
\newtheorem{Proposition}[Theorem]{Proposition}
\newtheorem{Definition}[Theorem]{Definition}
\newtheorem{Example}[Theorem]{Example}

\begin{document}

\title{Berezin-Type Operators on the\\ Cotangent Bundle of a Nilpotent Group \thanks{M. M. has been
 supported by the Fondecyt Project No. 1160359.}
}


\author{M. Mantoiu 
\medskip
}

\maketitle \vspace{-1cm}

\bigskip
\bigskip
{\bf Address}

\medskip
Departamento de Matem\'aticas, Universidad de Chile,

Las Palmeras 3425, Casilla 653, Santiago, Chile

\emph{E-mail:} mantoiu@uchile.cl

\bigskip

\maketitle

\begin{abstract}
We define and study coherent states, a Berezin-Toeplitz  quantization and covariant symbols on the product
 $\Xi:=\G\times\g^\sharp$ between a connected simply connected nilpotent Lie group and the dual of its
 Lie algebra. The starting point is a Weyl system codifying the natural Canonical Commutation Relations
 of the system. The formalism is meant to complement the quantization of the cotangent bundle 
$T^\sharp\G\cong\G\times\g^\sharp$ by pseudo-differential operators, to which it is connected in an 
explicit way. Some extensions are indicated, concerning $\tau$-quantizations and variable magnetic fields.
\end{abstract}

\section*{Introduction}\label{duci}

Trying to assign global pseudo-differential operators to large classes of locally compact groups
 $\G$\,, in \cite{MR} second countable, type $I$ unimodular groups have been treated, using 
operator-valued symbols defined on the product $\G\times\wG$\,, where the unitary dual $\wG$ is the 
family of equivalence classes of irreducible representations of $\G$\,. Much more can be said in 
particular cases: (a) for compact Lie groups \cite{RT} and (b) for graded nilpotent Lie groups \cite{FR1}. Actually the number of papers treating these two particular classes in great detail is growing fast. Keeping the same general framework as in \cite{MR}, in \cite{M} the related Berezin-type quantization has been explored. The operator-valuedness of the symbols and the fact that different irreducible representations act in different Hilbert spaces made the theory technically challenging.

\smallskip
In particular cases at least, one hopes for a simpler-looking quantization in terms of scalar-valued 
symbols. This is possible for connected simply connected nilpotent groups, due to some special properties,
 allowing finally to define a well-behaved Fourier transformation from functions (or distributions) 
defined on $\G$ to function (or distributions) defined on $\g^\sharp$, the dual of the Lie algebra
 $\g$\,. There is a drawback, however: this Fourier transformation does not intertwine  multiplication
 with convolution (and this is due to the fact that, while being a diffeomorphism, the exponential map 
does not have good algebraic properties).

\smallskip
A graded structure on the Lie  algebra surely helps. We mention some previous works \cite{BKG,Glo1,Glo2,Glo3,Ma1,Ma2,Me,Mi,Mi1},
 mainly dedicated to particular types of nilpotent groups or to invariant symbols (depending only 
on $\xi\in\g^\sharp$). In \cite{MR,MR1} quantization formulas for the general nilpotent case has been
 mentioned and connections with the operator-valued calculus on $\G\times\wG$ have been indicated. 
Actually the connection consists in combining together two different partial Fourier transformations. 
In the case of groups having (generic) square integrable irreducible representations modulo the center, 
the connection becomes nice and effective, involving Kirillov's theory and the Weyl-Pedersen 
pseudo-differential calculus on coadjoint orbits \cite{MR1}.

\smallskip
Anyhow, it is natural to introduce and study the generalization of the Berezin-Toeplitz 
(also called anti-Wick) formalism in the setting of the phase-space
 $\Xi:=\G\times\g^\sharp\cong T^\sharp\G$\,. The vector group $\G=\R^n$ is a guiding particular case. 
In spite of the mentioned connection between the pseudo-differential quantizations on 
$\G\times\wG$ and $\G\times\g^\sharp$, via a composition of partial Fourier transformations, 
{\it the Berezin-Toeplitz formalisms on the two "phase spaces" are not equivalent}. The reason is that
 their (weak) definitions involve in both cases products of {\it two} Fourier-Wigner functions 
(see \eqref{viespe} for instance), and one of the two Fourier transformations behaves badly with 
respect to multiplication. So there is no isomorphism between the objects from the present article 
and the analog ones from \cite{M}.

\smallskip
After fixing in section \ref{bogart} some notations and conventions about groups and Hilbert spaces, 
in section \ref{firica} we proceed to describe the basic operators acting in $\H:=L^2(\G)$ that will 
be the building blocks of our theory. They can be understood as global or as infinitesimal operations
verifying the Canonical Commutation Relations inherent to the pair $(\G,\g^\sharp)$\,.

\smallskip
From such building blocks, in section \ref{firicra} we construct the Weyl system, a highly non-commutative
 version of the usual one (phase-space shifts) in $\R^n$.
Due to the complexity of the Canonical Commutation Relations, it is not even a projective 
representation of the group $\G\times\g^\sharp$. So one cannot invoke directly results and techniques
 from the existing theory in group-form. The "matrix coefficients" of this Weyl system lead to a 
Fourier-Wigner transform, coherent states, the Bargmann transform, reproducing kernel Hilbert spaces, etc. 
(We use a certain terminology, especially by analogy with the $\R^n$-case; but even in this commutative 
case there are so many different denominations. So we do not expect all the readers to be satisfied with
 our choices.)

\smallskip
In section \ref{firton} one defines the Berezin quantization and study its basic properties. It is 
positive-preserving, it sends $L^p$ spaces of symbols into Schatten-von Neumann classes of order $p$ 
on $L^2(\G)$ and gets a Toeplitz form in the Bargmann representations. Some simple examples are included. 
Other more refined results are postponed to a future publication, mainly because they need first to 
establish a suitable coorbit (and modulation space) theory on $\G\times\g^\sharp$\,. 

\smallskip
The matrix elements of a bounded operator between coherent states define the covariant (lower) symbol. 
It is studied in section 5.  Among others, it provides some lower bounds for certain Schatten-von Neumann 
norms. Kernels of regular operators may be expressed in terms of the covariant symbols and the coherent 
states. Hopefully, this will be used in a future paper to prove a Beals-type criterion for pseudo-differential 
operators with scalar-valued symbols on $T^\sharp\G$\,.

\smallskip
Then we compute the pseudo-differential symbol of a Berezin operator; the correspondence is no longer 
given by a convolution, as in the standard case.

\smallskip
In a final section, we briefly indicate two extensions. First we treat $\tau$-quantizations related to 
ordering issues. We show how this may be implemented at the level of the basic objects. Then we describe
 what happens when a variable magnetic field is also present. For pseudo-differential operators this has
 been done in \cite{BM}. Here we put into evidence the changes needed in the Berezin theory.

\smallskip
Up to our knowledge, the results in this article are not contained in the existing literature. In 
particular, 
projective group representation methods do not apply. However, the constructions and proofs are inspired 
by other, different situations. We were mainly guided by the book \cite{Wo}. The related but not 
isomorphic theory from \cite{M} has also been valuable. The literature on coherent states, Berezin type 
(or localization) operators and related topics is huge; we only cite some references \cite{AAG,BBR,BOW,CGr,CR,GMP,Ha,TB,Wo,Wo1}. As said above, modulation spaces will
 be studied in such a framework subsequently and this will bring to our attention the expanding literature
 on time-frequency methods. Actions in $L^p$-spaces will also be investigated.

\section{Framework}\label{bogart}

The scalar products in a Hilbert space are linear in the first variable. For a given (complex, separable) 
Hilbert space $\H$, one denotes by $\mathbb B(\H)$ the $C^*$-algebra of all linear bounded operators in 
$\H$ and by $\mathbb B^p(\H)$ the bi-sided $^*$-ideal of all Schatten-von Neumann operators of exponent 
$p\ge 1$\,. In particular $\mathbb K(\H)\equiv\mathbb B^\infty(\H)$ is the $C^*$-algebra of all the 
compact operators in $\H$\,. The unitary elements of $\mathbb B(\H)$ form the group $\mathbb U(\H)$\,.

\smallskip
Let $\G$ be a connected simply connected nilpotent Lie group with unit $\e$\,, Haar measure $dx$ and 
unitary dual $\wG$\,. 
Let $\mathfrak g$ be the Lie algebra of $\G$ and $\mathfrak g^\sharp$ its dual. If $X\in\g$ and $\xi\in\g^\sharp$ we 
set $\<X\!\mid\!\xi\>\!:=\xi(X)$\,. We also denote by $\exp:\mathfrak g\to\G$ the exponential map, which is a diffeomorphism. Its inverse is denoted by $\log:\G\rightarrow\mathfrak g$\,. Under these diffeomorphisms the Haar measure on $\G$ corresponds to a Haar measure $dX$ on $\g$ (normalized accordingly).  For each $p\in[1,\infty]$\,, one has an isomorphism 
$$
L^p(\G)\overset{{\rm Exp}}{\longrightarrow} L^p(\mathfrak g)\,,\ {\rm Exp}(u):=u\circ\exp$$ 
with inverse 
$$
L^p(\mathfrak g)\overset{{\rm Log}}{\longrightarrow} L^p(\G)\,,\ {\rm Log}(\nu):=\nu\circ\log\,.
$$ 
The Schwartz spaces $\S(\G)$ and $\S(\g)$ are defined as in \cite[A.2]{CG}; they are isomorphic
 Fr\'echet spaces.

\smallskip
For $X,Y\in\mathfrak g$ we set 
\begin{equation*}\label{diosdado}
\begin{aligned}
X\bullet Y&:=\log[\exp(X)\exp(Y)]\\
&=X+Y+\frac{1}{2}[X,Y]+\frac{1}{12}\big([X,[X,Y]]+[Y,[Y,X]]\big)+\dots
\end{aligned}
\end{equation*}
It is a group composition law on $\mathfrak g$\,, given by a polynomial expression in $X,Y$ (the Baker-Campbel-Hausdorff formula). 
The unit element is $0$ and $X^{\bullet}\equiv -X$ is the inverse of $X$ with respect to $\bullet$\,.

\smallskip
There is a Fourier transformation, given by the duality $\big(\g,\g^\sharp\big)$\,, defined essentially  by
\begin{equation*}\label{clata}
\big(\mathcal F h\big)(\xi):=\int_{\g}e^{-i\<X\mid \xi\>} h(X)\,dX.
\end{equation*}
It is a linear topological isomorphism $\,\mathcal F:\mathcal S(\g)\rightarrow\mathcal S(\g^\sharp)$ and 
a unitary map $\,\mathcal F:L^2(\g)\rightarrow L^2(\g^\sharp)$\,.  
 Composing with the isomorphisms ${\rm Exp}$ and ${\rm Log}$ one gets Fourier transformations
\begin{equation*}\label{fericire}
\mathscr F:=\mathcal F\circ{\rm Exp}:\S(\G)\rightarrow \S(\g^\sharp)\,,\quad \mathscr F^{-1}\!:={\rm Log}\circ\mathcal F^{-1}\!:\S(\g^\sharp)\rightarrow \S(\G)\,,
\end{equation*}
\begin{equation*}\label{qlata}
(\mathscr F u)(\xi)=\int_{\g}e^{-i\<X\mid \xi\>}u(\exp X)dX=\int_\G e^{-i\<\log x\mid \xi\>} u(x)dx\,,
\end{equation*}
\begin{equation*}\label{qluta}
\big(\mathscr F^{-1}w\big)(x)=\int_{\g^\sharp}\!e^{i\<\log x\mid \xi\>}w(\xi)d\xi\,.
\end{equation*}
These maps also define unitary isomorphisms of the corresponding $L^2$-spaces.

\section{Canonical commutation relations}\label{firica}

One has the (strongly continuous) unitary representation ${\rm M}:(\g^\sharp,+)\to\mathbb U\big[L^2(\G)\big]$ given by
\begin{equation*}\label{generez}
[{\rm M}_\zeta(u)](x):=e^{i\<\log x\mid\zeta\>}u(x)\,.
\end{equation*}
If we denote by ${\sf Mult}(\psi)$ the operator of multiplication by functions $\psi$ defined on $\G$\,, 
one has 
\begin{equation*}\label{diverse}
{\sf M}_\zeta={\sf Mult}(\varepsilon_\zeta)={\sf Mult}\big(e^{i\lambda_\zeta}\big)\,,
\end{equation*} 
where we introduced the function
\begin{equation*}\label{lambda}
\lambda_\zeta:\G\to\R\,,\quad\lambda_\zeta(x):=\<\log x\!\mid\!\zeta\>
\end{equation*}
and its imaginary exponential
\begin{equation}\label{epsilon}
\varepsilon_\zeta:\G\to\T\subset\mathbb C\,,\quad\varepsilon_\zeta(x):=e^{i\<\log x\mid\zeta\>}.
\end{equation}
For each $\zeta\in\g^\sharp$ one has a $1$-parameter subgroup 
\begin{equation*}\label{equation}
\R\ni t\to{\sf M}_{t\zeta}=e^{it\Lambda_\zeta}\in\mathbb U\big[L^2(\G)\big]\,,
\end{equation*} 
with infinitesimal generator
\begin{equation*}\label{incapatanare}
\Lambda_\zeta:={\sf Mult}\big(\lambda_\zeta\big)={\sf Mult}(\<\log (\cdot)\!\mid\!\zeta\>)\,.
\end{equation*}

One also has {\it the left and the right unitary representations} 
$$
{\sf L,R}:(\G,\cdot)\to\mathbb U\big[L^2(\G)\big]\,,
$$ 
defined by
\begin{equation*}\label{leftright}
\big[{\sf L}_z(u)\big](x):=u\big(z^{-1}x\big)\,,\quad\big[{\sf R}_z(u)\big](x):=u(xz)\,.
\end{equation*}
For fixed $Z\in\g$\,, there are $1$-parameter subgroups
\begin{equation*}\label{thesubgrup}
\R\ni t\to {\sf L}_{\exp(tZ)}=e^{it(iD^{\sf L}_Z)},\quad\R\ni t\to {\sf R}_{\exp(tZ)}=e^{it(-iD^{\sf R}_Z)},
\end{equation*}
where
$$
\big[D^{\sf L}_Z(u)\big](x):=\frac{d}{dt}\Big\vert_{t=0}u\big(\exp[tZ]x\big)\,,\quad\big[D^{\sf R}_Z(u)\big](x):=\frac{d}{dt}\Big\vert_{t=0}u\big(x\exp[tZ]\big)\,.
$$

Note the "multiplication relations"
\begin{equation*}\label{CCR}
{\sf L}_y{\sf L}_z={\sf L}_{yz}\,,\quad {\sf M}_\eta {\sf M}_\zeta={\sf M}_{\eta+\zeta}\,,\quad {\sf L}_z {\sf M}_\zeta=e^{i\<\log(z^{-1}\cdot)-\log(\cdot)\mid\zeta\>}{\sf M}_\zeta {\sf L}_z\,.
\end{equation*}
and the "commutation relations" (on the Schwartz space $\S(\G)$, for instance)
\begin{equation*}\label{maiusoare}
\big[D^{\sf L}_Y,D^{\sf L}_Z\big]=D^{\sf L}_{[Y,Z]}\,,\quad\big[D^{\sf R}_Y,D^{\sf R}_Z\big]=D^{\sf R}_{[Y,Z]}\,,
\end{equation*}
\begin{equation}\label{maigrele}
\big[D^{\sf L}_Z,\Lambda_\zeta\big]={\sf Mult}\big(D^{\sf L}_Z\lambda_\zeta\big)\,,\quad\big[D^{\sf R}_Z,\Lambda_\zeta\big]={\sf Mult}\big(D^{\sf R}_Z\lambda_\zeta\big)\,.
\end{equation}
For concreteness, let us set ${\sf ad}_X(Z):=[X,Z]$ and compute (in the BCH formula, only the terms that are linear in $tZ$ contribute):
\begin{equation}\label{ce-oiesi}
\begin{aligned}
\big(D^{\sf L}_Z\lambda_\zeta\big)(\exp X)&=\frac{d}{dt}\Big\vert_{t=0}\<\log (\exp[tZ]\exp X)\!\mid\!\zeta\>\\
&=\frac{d}{dt}\Big\vert_{t=0}\<\,[tZ]\bullet X\!\mid\!\zeta\>\\
&=\Big\<Z-\frac{1}{2}{\sf ad}_X(Z)+\frac{1}{12}{\sf ad}^2_X(Z)+\dots\,\big\vert\,\zeta\Big\>\,.
\end{aligned}
\end{equation}
The sum is finite. Let us define {\it the infinitesimal coadjoint action}
\begin{equation*}\label{coi}
\gamma:\G\to{\sf Aut}(\g^\sharp)\,,\quad\gamma_x(\zeta)\equiv{\sf ad}_{-\log x}^\sharp(\zeta):=\zeta\circ{\sf ad}_{-\log x}\,.
\end{equation*}
Then \eqref{ce-oiesi} may be rewritten
\begin{equation*}\label{ceoiesi}
\big(D^{\sf L}_Z\lambda_\zeta\big)(x)=\Big\<Z\,\big\vert\,\zeta+\frac{1}{2}\gamma_x(\zeta)+\frac{1}{12}\gamma_x^2(\zeta)+\dots\Big\>\,.
\end{equation*}
This is a function of $x$\,, which becomes a constant $\<Z\!\mid\!\zeta\>$ precisely when the 
group $\G$ is Abelian. There is a similar formula for $D^{\sf R}_Z\lambda_\zeta$\,.

\section{Weyl systems, the Fourier-Wigner transform and coherent states}\label{firicra}

\begin{Definition}\label{sigmund}
For $(z,\zeta)\in\G\times\g^\sharp$ one defines a unitary operator ${\sf W}(z,\zeta):={\sf M}_\zeta{\sf L}_z$ in $L^2(\G)$ by
\begin{equation}\label{friedar}
\[{\sf W}(z,\zeta)u\]\!(x):=e^{i\<\log x\mid \zeta\>}u(z^{-1}x)\,,
\end{equation}
with adjoint
\begin{equation*}\label{frigider}
\[{\sf W}(z,\zeta)^*u\]\!(y):=e^{-i\<\log(zy)\mid \zeta\>}u(zy)\,.
\end{equation*}
\end{Definition}

This extends the notion of {\it Weyl system} (or {\it time-frequency shifts}) from the case $\G=\R^n$.
These operators also act as isomorphisms of the Schwartz space $\S(\G)$ and can be extended to 
isomorphisms of the space $\S'(\G)$ of tempered distributions. Note that they also define isometries in any $L^p(\G)$ space.

\begin{Lemma}\label{calcul}
For $(z,\zeta),(y,\eta)\in\G\times\g^\sharp$ one has
\begin{equation*}\label{fulppe}
{\sf W}(z,\zeta)\,{\sf W}(y,\eta)=\Gamma\big[(z,\zeta),(y,\eta)\big]\,{\sf W}(zy,\zeta+\eta)\,,
\end{equation*}
where $\Gamma\big[(z,\zeta),(y,\eta)\big]$ is the operator of multiplication by the function
\begin{equation*}\label{sireata}
x\mapsto\gamma\big[(z,\zeta),(y,\eta);x\big]=\exp\big\{-i\<\,\log x-\log(z^{-1}x)\mid \eta\,\>\big\}\,.
\end{equation*}
\end{Lemma}

Thus the Weyl system is very far from being a projective representation.

\begin{Definition}\label{wigner}
For $\,u,v\in\H:=L^2(\G)$ one sets $\mathscr W_{u,v}\equiv\mathscr W_{u\otimes\overline v}:\G\times\g^\sharp\to\mathbb C$ by
\begin{equation}\label{vigner}
\mathscr W_{u,v}(z,\zeta):=\<{\sf W}(z,\zeta)u,v\>=\int_\G e^{i\<\log y\mid\zeta\>}u(z^{-1}y)\overline{v(y)}dy\,.
\end{equation}
and call it {\rm the Fourier-Wigner transform}.
\end{Definition}

\begin{Lemma}\label{startortog}
The Fourier-Wigner transform extends to a unitary map 
\begin{equation*}\label{unicra}
\mathscr W\!:\H\otimes\overline\H\cong L^2(\G\times\G)\to L^2(\G\times\g^\sharp)\,.
\end{equation*}
It also defines isomorphisms
\begin{equation*}\label{treica}
\mathscr W\!:\S(\G)\,\overline\otimes\,\S(\G)\cong\S(\G\times\G)\to \S(\G\times\g^\sharp)\,.
\end{equation*}
\begin{equation*}\label{patrica}
\mathscr W\!:\S'(\G)\,\overline\otimes\,\S'(\G)\cong\S'(\G\times\G)\to \S'(\G\times\g^\sharp)\,.
\end{equation*}
\end{Lemma}

\begin{proof}
The map $\W=({\rm id}\otimes\mathscr F^{-1})\circ{\rm C}$ is composed of a partial Fourier transform and 
a unitary change of variables $(x,y)\to{\rm C}(x,y):=\big(x^{-1}y,y\big)$, and this leads easy to a proof
 of all the assertions. 
\end{proof}

In particular, one has {\it the orthogonality relations}:
\begin{equation}\label{orthog}
\big\<\mathscr W_{u,v},\mathscr W_{u'\!,v'}\big\>_{L^2(\G\times\g^\sharp)}=\<u,u'\>\<v',v\>\,.
\end{equation}

\begin{Definition}\label{coeres}
For some fixed $L^2$-normalized $\o\in\S(\G)\subset L^2(\G)\equiv\H$ and for  every $(z,\zeta)\in\G\times\g^\sharp$, we define {\rm the coherent state} $\,\o_{z,\zeta}\!:={\sf W}(z,\zeta)^*\o$\,; explicitly
\begin{equation*}\label{eltita}
\o_{z,\zeta}(x)=e^{-i\<\log(zx)\mid \zeta\>}\o(zx)\,.
\end{equation*} 
\end{Definition}

The associated rank one projector is given by
\begin{equation}\label{rankone}
{\O_{z,\zeta}}(u):=\big\<u,\o_{z,\zeta}\big\>\,\o_{z,\zeta}=\mathscr W_{u,\o}(z,\zeta)\,\o_{z,\zeta}\,,\quad\forall\,u\in\H\,.
\end{equation}
It is an integral operator with kernel $\varpi_{z,\zeta}:=\o_{z,\zeta}\otimes\overline{\o_{z,\zeta}}\in S(\G\times\G)$ (or in $L^2(\G\times\G)$ more generally for $\o$ in $L^2(\G)$)\,. 

\smallskip
{\it The canonical (or modulation) mapping} associated to the vector $\o$ (or {\it the generalized 
Bargmann transformation}) 
\begin{equation*}\label{metrie}
\mathscr B_{\o}:\H\to L^2(\G\times\g^\sharp)\,,\quad\mathscr B_{\o}(u):=\mathscr W_{u,\o}
\end{equation*} 
is an isometry with adjoint
\begin{equation*}\label{cumetrie}
\mathscr B_{\o}^\dag:L^2(\G\times\g^\sharp)\to \H\,,\quad\mathscr B^\dag_{\o}(h):=\int_\G\int_{\g^\sharp}h(z,\zeta)\,\o_{z,\zeta}\,dzd\zeta\,.
\end{equation*} 
The isometry condition may be seen as {\it an inversion formula}:
\begin{equation}\label{invform}
u=\int_\G\int_{\g^\sharp}\big\<u,\o_{z,\zeta}\big\>\,\o_{z,\zeta}\,dzd\zeta\,.
\end{equation}
The final projection $\mathscr P_{\o}\!:=\mathscr B_{\o}\mathscr B_{\o}^\dag$ is an integral operator with kernel 
\begin{equation}\label{adica}
p_{\o}(z,\zeta;z',\zeta'):=\big\<\o_{z,\zeta},\o_{z'\!,\zeta'}\big\>\,.
\end{equation}
One also have {\it the reproducing formula} $\mathscr B_{\o}(u)=\big(\mathscr B_{\o}\mathscr B^{\dag}_{\o}\mathscr B_{\o}\big)(u)=\mathscr P_{\o}\!\[\mathscr B_{\o}(u)\]$\,, i.e.
\begin{equation*}\label{tauras}
\[\mathscr B_{\o}(u)\](x,\xi)=\int_\G\int_{\g^\sharp} \<\o_{x,\xi},\o_{z,\zeta}\>\[\mathscr B_{\o}(u)\]\!(z,\zeta)\,dzd\zeta\,.
\end{equation*}
Thus $\mathscr P_{\o}\!\[L^2(\G\times\g^\sharp)\]$ is a reproducing kernel Hilbert space with 
reproducing kernel $p_{\o}$\,,
composed of bounded continuous functions on $\G\times\g^\sharp$.

\section{The Berezin quantization}\label{firton}

Occasionally, we are going to use notations as $\X\!:=(x,\xi),\Y\!:=(y,\eta),\Z\!:=(z,\zeta)\in\Xi\!:=\G\times\g^\sharp$, with product measure $d\X:=dxd\xi$\,. Actually, both types of notations will be used alternatively.
We denote by $\<\cdot,\cdot\>_{(\Xi)}$ both the $L^2(\Xi)$-inner product and the various related 
duality forms (as $\S(\Xi)\times\S'(\Xi)\to\mathbb C$ for example). The precise meaning will be specified 
or will be obvious from the context.

\begin{Definition}\label{uwe}
Let $\,\o\in \S(\G)$ be a fixed $L^2$-normalized vector. We define formally the operator in $L^2(\G)$
\begin{equation}\label{albina}
{\sf Ber}_\o(f):=\int_\G\int_{\g^\sharp}\!f(x,\xi)\,\O_{x,\xi}\,dxd\xi=\int_\Xi\!f(\X)\,\O_{\X}\,d\X,
\end{equation}
where $\O_\X$ is defined in \eqref{rankone}, and call it {\rm the Berezin operator associated to the symbol $f$ and the vector $\o$}\,.
\end{Definition}

This should be taken in weak sense: taking \eqref{vigner} and \eqref{rankone} into account, for 
any $u,v\in L^2(\G)$ one gets
\begin{equation}\label{viespe}
\begin{aligned}
\big\<{\sf Ber}_\o(f)u,v\big\>:=&\int_\G\int_{\g^\sharp}\!f(x,\xi)\,\big\<\O_{x,\xi}(u),v\big\>\,dxd\xi\\
=&\int_\G\int_{\g^\sharp}f(x,\xi)\,\W_{u,\o}(x,\xi)\,\overline{\W_{v,\o}(x,\xi)}\,dxd\xi\\
=&\,\big\<f,\overline{\W_{u,\o}}\,\W_{v,\o}\big\>_{(\Xi)}\,.
\end{aligned}
\end{equation}
This allows many different (but compatible) interpretations, based on the properties of the Fourier-Wigner transform. For instance, if $u,v\in\S(\G)$\,, the last term of \eqref{viespe} makes sense for $f\in\S'(\G\times\g^\sharp)$ and one gets a linear continuous operator 
\begin{equation*}\label{cofilae}
{\sf Ber}_\o(f):\S(\G)\to\S'(\G)\,,\quad f\in\S'(\G\times\g^\sharp)\,. 
\end{equation*}
For similar reasons, ${\sf Ber}_\o(f):\S'(\G)\to\S(\G)$ is well-defined, linear and continuous if $f\in\S(\G\times\g^\sharp)$\,.
If  $u,v\in L^2(\G)$\,, then $\W_{u,\o},\W_{v,\o}\in L^2(\Xi)$\,, thus $\W_{u,\o}\overline{\W_{v,\o}}\in L^1(\Xi)$ and one gets
\begin{equation*}\label{caibalae}
{\sf Ber}_\o(f)\in\mathbb B\big[L^2(\G)\big]\,,\quad f\in L^\infty(\G\times\g^\sharp)\,. 
\end{equation*}
It is obvious that ${\sf Ber}_\o(f)^*\!={\sf Ber}_\o(\overline f)$ and that {\it ${\sf Ber}_\o(f)$ is a 
positive operator in $L^2(\G)$ if  $\,f\in L^\infty(\Xi)$ is (almost everywhere) positive}. By the 
orthogonality relations \eqref{orthog} one may write
$$
\big\<{\sf Ber}_\o(1)u,v\big\>=\int_{\Xi}\!\W_{v,\o}(\X)^*\,\W_{u,\o}(\X)d\X=\big\<\W_{u,\o}(\X),\W_{v,\o}(\X)\big\>_{(\Xi)}\!=\<u,v\>\,,
$$
implying that ${\sf Ber}_\o(1)=1_{L^2(\G)}$\,. 

\smallskip
We gather other important properties of the Berezin operators in connection with the Schatten-von Neumann classes in the next result. 

\begin{Theorem}\label{bondar}
For every $s\in[1,\infty]$ one has a linear bounded map ${\sf Ber}_\o:L^{s}(\Xi)\to\mathbb B^s\big[L^2(\G)\big]$ satisfying
\begin{equation}\label{aceea}
\p\!{\sf Ber}_\o(f)\!\p_{\mathbb B^s[L^2(\G)]}\,\le\;4^{1/s}\p\!f\!\p_{L^s(\Xi)}.
\end{equation} 
In particular, if $\,f\in L^{1}(\Xi)$\,, then ${\sf Ber}_\o(f)$ is a trace-class operator with
\begin{equation*}\label{musca}
{\rm Tr}\big[{\sf Ber}_\o(f)\big]=\int_\G\int_{\g^\sharp}f(x,\xi)\,dxd\xi\,.
\end{equation*}
\end{Theorem}

\begin{proof}
For $s=\infty$ we write using the definitions, an obvious $L^1-L^\infty$ estimate, the Cauchy-Schwartz 
inequality and the orthogonality relation
$$
\begin{aligned}
\p\!{\sf Ber}_\o(f)\!\p_{\mathbb B[L^2(\G)]}&=\sup_{\p u\p=1=\p v\p}\big\vert\big\<{\sf Ber}_\o(f)u,v\big\>\big\vert\\
&=\sup_{\p u\p=1=\p v\p}\big\vert\big\<f,\overline{\W_{u,\o}}\,\W_{v,\o}\big\>_{(\Xi)}\big\vert\\
&\le\;\p\!f\!\p_{L^{\infty}}\!\sup_{\p u\p=1=\p v\p}\!\p\!\overline{\W_{u,w}}\,\W_{v,\o}\!\p_{L^{1}}\\
&\le\;\p\!f\!\p_{L^{\infty}}\sup_{\p u\p=1}\!\p\!\W_{u,\o}\!\p_{L^{2}}\!\sup_{\p v\p=1}\!\p\!\W_{v,\o}\!\p_{L^{2}}\\
&=\;\p\!f\!\p_{L^{\infty}}.
\end{aligned}
$$
There is a version of the computation above showing that ${\sf Ber}_\o(f)$ is in fact also bounded 
if $f\in L^{s}(\Xi)$\,. It is based on complex interpolation, the H\"older inequality and improved 
properties of the Fourier-Wigner transformation, having as starting point the simple estimate
$$
|\mathscr W_{u,v}(\X)|=|\<{\sf W}(\X)u,v\>|\le\,\p\!u\!\p\p\!v\!\p\,,\quad\forall\ \X\in\Xi\,.
$$
But one needs the finer result \eqref{aceea}, in terms of Schatten-von Neumann classes.

\smallskip
We deal first with the trace class properties of the Berezin operator, assuming that $f\in L^1(\Xi)$ is 
positive\,. The connected simply connected nilpotent group $\G$ is second countable, so the Hilbert 
space $L^2(\G)$ is separable.  If $\{w_k\}_{k\in\N}$ is an orthonormal basis in $L^2(\G)$\,, one has 
by \eqref{viespe}, \eqref{rankone} and the Parseval identity 
$$
\begin{aligned}
{\rm Tr}\big[{\sf Ber}_\o(f)\big]&=\sum_k\big\<{\sf Ber}_\o(f)w_k,w_k\big\>\\
&=\sum_k\int_{\Xi}\!f(\X)\,\big\<\O_{\X}(w_k),w_k\big\>\,d\X\\
&=\sum_k\int_{\Xi}\!f(\X)\,\big\<\o_\X,w_k\big\>\,\big\<w_k,\o_\X\big\> \,d\X\\
&=\int_{\Xi}\!f(\X)\,\sum_k\,\big\<\o_\X,w_k\big\>\,\big\<w_k,\o_\X\big\>\,d\X\\
&=\int_{\Xi}\!f(\X)\,\big\<\o_\X,\o_\X\big\> \,d\X\\
&=\int_{\Xi}\!f(\X)\,d\X\,.
\end{aligned}
$$
If $f$ is positive, we already know that ${\sf Ber}_\o(f)$ is also positive and its trace norm is 
computed above:
$$
\p\!{\sf Ber}_\o(f)\!\p_{\mathbb B^1[L^2(\G)]}\,={\rm Tr}\big[{\sf Ber}_\o(f)\big]=\,\p\!f\!\p_{L^{1}}.
$$
One obtains the $s=1$ case of \eqref{aceea} for general $f$ by writing $f=\re[f]_+-\re[f]_-+i\im[f]_+-i\im[f]_-\,.$

\smallskip
The general case in \eqref{aceea} then follows by interpolation of order $\th=1/s$ from the cases $s=1$ 
and $s=\infty$\,, because 
one has $\Big[L^{\infty}(\Xi),L^{1}(\Xi)\Big]_{1/s}\!=L^{s}(\Xi)$ and $\Big[\mathbb B\big[L^2(\G)],\mathbb B^{1}\big[L^2(\G)\big]\Big]_{1/s}\!=\mathbb B^{s}\big[L^2(\G)\big]$\,.
\end{proof}

One can improve the constant $4^{1/s}$ to $1$\, by imitating arguments from \cite[Ch.\,14]{Wo} or 
from \cite{CT}.

\smallskip
Besides the compactness results following directly from Theorem \ref{bondar}, one also gets by 
approximation

\begin{Corollary}\label{cevine}
If $f\in C_0(\G\times\g^\sharp)$\,, then ${\sf Ber}_\o(f)$ is a compact operator in $L^2(\G)$\,.
\end{Corollary}

\begin{proof}
This is true for continuous compactly supported functions, by the result above, and then follows for 
every continuous function small at infinity, by uniform approximation and the case $s=\infty$ of 
\eqref{aceea}.
\end{proof}

\begin{Example}\label{deosebire}
{\rm Theorem \ref{bondar} supplies plenty of compact Berezin operators with symbols not belonging to
 $L^\infty(\Xi)$\,. In addition, we have $\O_{\X}={\sf Ber}_\o(\delta_{\X})$ (if \eqref{albina} seems 
too formal, one can easily compute with \eqref{viespe}), and this is a rank one projection defined by 
a distribution.
}
\end{Example}

\begin{Example}\label{deosebure}
{\rm  For $f\!:=\varphi\otimes 1$\,, where $\varphi:\G\to\mathbb C$\,, a short computation shows 
that ${\sf Ber}_\o(\varphi\otimes 1)$ is the operator of multiplication by the function $x\to\big(\check{\varphi}\star|\o|^2\big)(x):=\int_\G\varphi(z)|\o(zx)|^2 dz$\,, where $\check{\varphi}(y):=\varphi(y^{-1})$.
}
\end{Example}

\begin{Example}\label{deosebare}
{\rm  For $f\!:=1\otimes\psi$\,, where $\psi:\g^\sharp\to\mathbb C$\,, a short computation shows 
that ${\sf Ber}_\o(1\otimes\psi)$ is an integral operator with kernel 
\begin{equation*}\label{knukleul}
\big[h_\o(\psi)\big](x,y):=\int_\G \tilde\psi\,[\log(zx)-\log(zy)]\,\o(zx)\overline{\o(zy)}\,dz\,,
\end{equation*}
written in terms of the usual Fourier transform $\tilde\psi$ of $\psi$ attached to the duality 
$(\g,\g^\sharp)$\,.
}
\end{Example}

\begin{Proposition}\label{crooks}
For any $z\in\G$ and (for example) $f\in L^\infty(\G\times\g^\sharp)$\,, one has
\begin{equation*}\label{cuvar}
{\sf L}_z^*{\sf Ber}_\o(f){\sf L}_z={\sf Ber}_\o\big[f(\cdot z^{-1}\!,\cdot)\big]\,.
\end{equation*}
\end{Proposition}

\begin{proof}
One computes
$$
\begin{aligned}
{\sf L}_z^*{\sf Ber}_\o(f){\sf L}_z u&=\int_\G\int_{\g^\sharp}f(x,\xi)\,{\sf L}_z^*\O_{x,\xi}{\sf L}_z u\,dxd\xi\\
&=\int_\G\int_{\g^\sharp}f(x,\xi)\,\big\<u,{\sf L}_z^*{\sf W}(x,\xi)^*\o\big\>{\sf L}_z^*{\sf W}(x,\xi)^*\o\,dxd\xi\\
&=\int_\G\int_{\g^\sharp}f(x,\xi)\,\big\<u,{\sf W}(xz,\xi)^*\o\big\>{\sf W}(xz,\xi)^*\o\,dxd\xi\\
&=\int_\G\int_{\g^\sharp}f(x,\xi)\,\O_{xz,\xi}(u)\,dxd\xi\,,
\end{aligned}
$$
and then a change of variables leads to the result.
\end{proof}

The formula for ${\sf W}(z,\zeta)^*{\sf Ber}_\o(f){\sf W}(z,\zeta)$ is rather involved, due to 
Lemma \ref{calcul}.

\smallskip
We provide now a Toeplitz-like form of the operator ${\sf T\!p}_\o(f)\!:=\mathscr B_{\o}\circ{\sf Ber}_\o(f)\circ\mathscr B_{\o}^{\dag}$ living in $L^2\big(\G\times\g^\sharp\big)$\,.

\begin{Proposition}\label{carabus}
One has
\begin{equation}\label{ragace}
{\sf T\!p}_{\o}(f)=\mathscr P_\o\circ{\sf Mult}(f)\circ \mathscr P_\o\,,
\end{equation}
where ${\sf Mult}(f)$ is the point-wise multiplication by $f\in L^\infty(\Xi)$\,.
\end{Proposition}

\begin{proof}
Clearly \eqref{ragace} is equivalent to ${\sf Ber}_\o(f)=\mathscr B_\o^{\dag}\!\circ{\sf Mult}(f)\circ\mathscr B_\o$\,. For $u,v\in L^2(\G)$ we have
$$
\begin{aligned}
\big\<{\sf Ber}_{\o}(f)u,v\big\>&=\big\<f,\overline{\mathscr B_\o(u)}\,\mathscr B_\o(v)\big\>_{(\Xi)}\\
&=\int_\Xi\big[\mathscr B_\o(u)\big](\X)f(\X)\overline{\big[\mathscr B_\o(v)\big](\X)}\,d\X\\
&=\int_\Xi\Big({\sf Mult}(f)\big[\mathscr B_\o(u)\big]\Big)(\X)\,\overline{\big[\mathscr B_\o(v)\big](\X)}\,d\X\\
&=\Big\<{\sf Mult}(f)\big[\mathscr B_\o(u)\big],\mathscr B_\o(v)\Big\>_{\!(\Xi)}\\
&=\Big\<\big[\mathscr B_\o^{\dag}\!\circ{\sf Mult}(f)\circ\mathscr B_\o\big]u,v\Big\>_{\!(\Xi)}
\end{aligned}
$$
and the Proposition is proved.
\end{proof}

It follows immediately that ${\sf T\!p}_\o(f)$ is an integral operator with kernel
\begin{equation}\label{notas}
\big[t_\o(f)\big](\X,\Y):=\int_\Xi f(\Z)\<\o_\X,\o_\Z\>\<\o_\Z,\o_\Y\> d\Z\,.
\end{equation}

\section{The covariant symbol and the Berezin transform}\label{firika}

\begin{Definition}\label{tranteste}
{\rm The covariant symbol} ${\rm cov}_\o(T):\Xi\times\Xi\to\mathbb C$ of an operator $T\in\mathbb B\big[L^2(\G)\big]$ is 
\begin{equation*}\label{cavarice}
\big[{\rm cov}_\o(T)\big](\X,\X'):=\big\<T\o_{\X},\o_{\X'}\big\>=\big\<{\sf W}(\X')T{\sf W}(\X)^*\o,\o\big\>\,.
\end{equation*}
For the diagonal version we are also going to use the notation 
\begin{equation*}\label{vidoq}
\big[{\rm Cov}_\o(T)\big](\X):=\big[{\rm cov}_\o(T)\big](\X,\X)={\rm Tr}\big[T\O_\X\big]\,.
\end{equation*}
\end{Definition}

Clearly {\it ${\rm cov}_\o\!:\mathbb B\big[L^2(\G)\big]\to L^\infty(\Xi\times\Xi)$ is a linear
contraction and any ${\rm cov}_\o(T)$ is actually a continuous function}. Under further requirements 
on $\o$\,, it might have further regularity properties. For instance, if $\o\in\S(\G)$ (as we usually 
assume), then ${\rm cov}_\o(T)$ is smooth, with bounded derivatives.

\smallskip
Recall the composition of integral kernels
\begin{equation*}\label{yes}
(F\,\square\,G)(\X,\Y):=\int_\Xi F(\X,\Z)G(\Z,\Y)d\Z
\end{equation*}
and the adjoint
\begin{equation*}\label{manfredmann}
F^\square(\X,\Y)=\overline{F(\Y,\X)}\,.
\end{equation*}

\begin{Proposition}\label{pinkfloyd}
One has
\begin{equation*}\label{jethrotull}
{\rm cov}_\o(ST)={\rm cov}_\o(T)\,\square\,{\rm cov}_\o(S)\,,\quad{\rm cov}_\o(T^*)={\rm cov}_\o(T)^\square.
\end{equation*}
\end{Proposition}

\begin{proof}
By using the definitions and the inversion formula \eqref{invform} one gets
$$
\begin{aligned}
\big[{\rm cov}_\o(ST)\big](\X,\Y)&=\big\<T\o_\X,S^*\o_\Y\big\>\\
&=\int_\Xi \big\<T\o_\X,\o_\Z\big\>\big\<\o_\Z,S^*\o_\Y\big\>d\Z\\
&=\int_\Xi \big\<T\o_\X,\o_\Z\big\>\big\<S\o_\Z,\o_\Y\big\>d\Z\\
&=\int_\Xi \big[{\rm cov}_\o(T)\big](\X,\Z)\big[{\rm cov}_\o(S)\big](\Z,\Y)d\Z\\
&=\big[{\rm cov}_\o(T)\,\square\,{\rm cov}_\o(S)\big](\X,\Y)\,.
\end{aligned}
$$
The formula for the adjoint is obvious. Taking diagonal values one gets
\begin{equation*}\label{derrumbe}
{\rm Cov}_\o(T^*)=\overline{{\rm Cov}_\o(T)}\,.
\end{equation*}
\end{proof}

The diagonal covariant symbol provides lower bounds for the operator trace norm.

\begin{Proposition}\label{caceres}
If $\,T\in\mathbb B^1\!\big[L^2(\G)\big]$ then
\begin{equation*}\label{sarcastoc}
\p\!{\rm Cov}_\o(T)\!\p_{L^1(\G)}\,\le\,\p\!T\!\p_{\mathbb B^1}.
\end{equation*}
\end{Proposition}

\begin{proof}
The trace-class operator $T$ admits the strongly convergent representation
\begin{equation*}\label{representation}
T=\sum_{k=1}^\infty s_k(T)\<\cdot,\varphi_k\>\psi_k\,,
\end{equation*}
in terms of the (positive) singular values of $T$ and two orthonormal families. 
For every $\X\in\Xi$ we have
$$
\begin{aligned}
\big\vert\big[{\rm Cov}_\o(T)\big](\X)\big\vert&=\big\vert\<T\o_\X,\o_\X\>\big\vert\\
&=\Big\vert\sum_{k=1}^\infty s_k(T)\<\o_\X,\varphi_k\>\<\psi_k,\o_\X\>\Big\vert\\
&\le\frac{1}{2}\sum_{k=1}^\infty s_k(T)\Big(\vert\<\o_\X,\varphi_k\>\vert^2+\vert\<\psi_k,\o_\X\>\vert^2\Big)\,,
\end{aligned}
$$
implying
$$
\begin{aligned}
\p\!{\rm Cov}_\o(T)\!\p_{L^1(\G)}&\le\frac{1}{2}\sum_{k=1}^\infty s_k(T)\,\Big(\int_\Xi\vert\<\o_\X,\varphi_k\>\vert^2d\X+\int_\Xi\vert\<\psi_k,\o_\X\>\vert^2d\X\Big)\,.
\end{aligned}
$$
By the inversion formulas and by the normalization of the vectors, the two integrals equal $1$\,, and 
the remaining factor is the trace norm of the operator.
\end{proof}

By interpolation one readily gets

\begin{Corollary}\label{rodica}
If $\,T\in\mathbb B^p\big[L^2(\G)\big]$\,, with $p\in[1,\infty]$\,, then
\begin{equation*}\label{sarcastroc}
\p\!{\rm Cov}_\o(T)\!\p_{L^p(\G)}\,\le\,\p\!T\!\p_{\mathbb B^p}.
\end{equation*}
\end{Corollary}

\begin{Proposition}\label{zandowal}
If $\,T$ is a compact operator, ${\rm Cov}_\o(T)\in C_0(\Xi)$\,, i.\,e.\,it is a continuous function
 converging to zero at infinity.
\end{Proposition}

\begin{proof}
Continuity has already been mentioned. One still has to show that
\begin{equation*}\label{stillike}
\underset{\X\to\infty}{\lim}\<T{\sf W}(\X)\o,{\sf W}(\X)\o\>=0\,.
\end{equation*}
The operator $T$ being compact, it turns weak convergence into norm convergence. Also using the density of $\S(\G)$ in $L^2(\G)$ and the unitarity of the Weyl system, we are thus reduced to showing that
$$
\W_{\o,v}(\X)=\<{\sf W}(\X)\o,v\>\underset{\X\to\infty}{\longrightarrow}0\,,\quad\forall\,v\in S(\G)\,.
$$
This is obvious from the fact that $\o\in\S(\G)$ and that $\W$ is the composition between a change of 
variables and a partial Fourier transform. If $\o$ is only square integrable, one can still finish the 
proof by density and approximation.   
\end{proof}

\begin{Proposition}\label{cofi}
For every $f\in L^1(\G\times\g^\sharp)$ one has in terms of the kernel \eqref{notas} of the Toeplitz 
operator
\begin{equation}\label{deaceea}
{\rm cov}_\o\big({\sf Ber}_{\o}(f)\big)=t_\o(f)\,,
\end{equation}
with the particular case ({\rm the Berezin transform})
\begin{equation}\label{deacealalta}
\big[{\sf BT}_\o(f)\big](\X):=\big[{\rm Cov}_\o\big({\sf Ber}_{\o}(f)\big)\big](\X)=\int_\Xi\!f(\Z)\,\big\vert\big\<\o_\X,\o_\Z\big\>\big\vert^2 d\Z\,,
\end{equation}
and
\begin{equation}\label{integrasa}
\int_\Xi\big[{\sf BT}_\o(f)\big](\X)\,d\X=\int_\Xi f(\X)\,d\X\,.
\end{equation}
\end{Proposition}

\begin{proof}
Checking \eqref{deaceea} (and thus \eqref{deacealalta}) is an easy direct verification. Then proving
 \eqref{integrasa} relies on the formula 
\begin{equation}\label{vormula}
\int_\Xi|\<\o_\X,\o_\Z\>|^2d\X=1\,,\quad\forall\,\Z\in\Xi\,.
\end{equation}
Recalling the kernel \eqref{adica} of the {\it projection} $\mathscr P_{(\o)}$ and the normalization 
of $\o$\,, \eqref{vormula} becomes obvious. One may also use \eqref{invform} directly.
\end{proof}

Let us say that the operator $T\in\mathbb B\big[L^2(\G)\big]$ is {\it regularizing} if it extends to a 
continuous operator $T:\S'(\G)\to\S(\G)$\,; then it will have a kernel belonging to $\S(\G\times\G)$\,. 
This kernel may be expressed in terms of the covariant symbol and the coherent states.

\begin{Proposition}\label{nucleuldansei}
The kernel $K_T:\G\times\G\to\mathbb C$ of the regularizing operator $T$ is given through the formula
\begin{equation*}\label{eciueishn}
K_T(x,y)=\int_\Xi\int_\Xi\big[{\rm cov}_\o(T)\big](\Z,\Z')\,\o_{\Z'}(x)\,\overline{\o_\Z}(y)\,d\Z\,d\Z'.
\end{equation*}
\end{Proposition}

\begin{proof}
Computing for $u\in\S(\G)$ (for instance),  we are going to use the inversion formula twice:
$$
\begin{aligned}
(Tu)(x)&=\int_\Xi\big\<Tu,\o_{\Z'}\big\>\,\o_{\Z'}(x)\,d\Z'\\
&=\int_\Xi\big\<u,T^*\o_{\Z'}\big\>\,\o_{\Z'}(x)\,d\Z'\\
&=\int_\Xi\Big\<u,\int_\Xi\<T^*\o_{\Z'},\o_\Z\big\>\,\o_\Z\,d\Z\Big\>\,\o_{\Z'}(x)\,d\Z'\\
&=\int_\Xi\int_\Xi\,\overline{\<T^*\o_{\Z'},\o_\Z\big\>}\big\<u,\o_\Z\big\>\,\o_{\Z'}(x)\,d\Z d\Z'\\
&=\int_{\G}\int_\Xi\int_\Xi\big\<T\o_\Z,\o_{\Z'}\big\>\,\o_{\Z'}(x)\,\overline{\o_\Z(y)}\,u(y)dy\,d\Z\,d\Z'\\
&=\int_{\G}K_{T}(x,y)u(y)dy=\big[{\sf Int}(K_T)u\big](x)\,.
\end{aligned}
$$
\end{proof}

\section{Connection with pseudo-differential operators}\label{tronol}

One defines {\it the pseudo-differential operator with symbol $a:\G\times\g^\sharp\to\mathbb C$} by the 
formula
\begin{equation}\label{dulcisor}
\big[{\sf Op}(a)u\big]\!(x):=\int_\G\int_{\g^\sharp}\!e^{i\<\log(xy^{-1})\mid \xi\>}a(x,\xi)u(y)\,dyd\xi\,.
\end{equation}
It is an integral operator with kernel ${\sf K}_a(x,y):=\int_{\g^\sharp}\!e^{i\<\log(xy^{-1})\mid \xi\>}a(x,\xi)d\xi$\,.
The structure of this kernel (obtained from the symbol $a$ by a partial Fourier transform and a change of
 variables) allows various types of interpretation of the formula \eqref{dulcisor} and leads to the 
properties of the quantization $\sf Op$\,, that we do not discuss here in detail. Examining this kernel, 
one sees for instance that \eqref{dulcisor} defines a unitary mapping ${\sf Op}:L^2\big(\G\times\g^\sharp\big)\rightarrow\mathbb B^2\!\[L^2(\G)\]$. Versions involving Schwartz spaces are also easy to obtain. Note that the Weyl system \eqref{friedar} can be recuperated as 
\begin{equation*}\label{newaige}
{\sf W}(z,\zeta)={\sf Op}(\epsilon_{z,\zeta})\,,\quad{\rm where}\quad\epsilon_{z,\zeta}(x,\xi):=e^{i\<\log x\mid \zeta\>}e^{-i\<\log z\mid \xi\>}
\end{equation*}
(see also \eqref{epsilon}) and that the Fourier-Wigner transform \eqref{vigner} may also be involved in 
the definition of $\sf Op$\,. If $a$ only depends on $x$ then $\sf Op$ is a multiplication operator, 
while if $a$ only depends on $\xi$\,, $\sf Op$ becomes a (left) convolution operator. 

\begin{Proposition}\label{traznaye}
Suppose (say) that $f\in\S(\G\times\g^\sharp)$\,. The Berezin operator ${\sf Ber}_{\o}(f)$ is a 
pseudo-differential operator with symbol
\begin{equation}\label{ghent}
\begin{aligned}
\big[a_\o(f)\big](x,\xi):=\int_\G\int_\G\int_{\g^\sharp}& e^{-i\<\log y\mid\xi\>}\,e^{i\<\log(zy^{-1}x)-\log(zx)\mid\zeta\>}\\
&f(z,\zeta)\,\o(zx)\,\overline{\o(zy^{-1}x)}\,dy dz d\zeta\,.
\end{aligned}
\end{equation}
\end{Proposition}

\begin{Remark}\label{eren}
In the Abelian case $\G=\R^n$ \eqref{ghent} simply reduces to a convolution:
\begin{equation*}\label{ghiont}
\begin{aligned}
\big[a_\o(f)\big](x,\xi)&=\int_{\R^n}\!\int_{\R^n}\!\int_{\R^n} f(z,\zeta)\,e^{-i\<y\mid\zeta+\xi\>}\o(z+x)\,\overline{\o(z-y+x)}\,dy dz d\zeta\\
&=\int_{\R^n}\!\int_{\R^n}\Big[\int_{\R^n} \,e^{-i\<y\mid\eta\>}\o(s)\,\overline{\o(s-y)}\,dy\Big]f(s-x,\eta-\xi)ds d\eta\,.
\end{aligned}
\end{equation*}
\end{Remark}

\begin{proof}
As said above, ${\sf Op}(a)$ is an integral operator with kernel $\,{\sf K}_{a}:\G\times\G\rightarrow\mathbb C$ given by
\begin{equation*}\label{acrisor}
{\sf K}_{a}(x,y)=\int_{\g^\sharp}\!e^{i\<\log(xy^{-1})\mid \xi\>}a(x,\xi)d\xi=\big[\big({\rm id}\otimes\mathscr F^{-1}\big)a\big]\big(x,xy^{-1}\big)\,.
\end{equation*}
The symbol may be recovered from the kernel by means of the formula
\begin{equation}\label{recup}
a(x,\xi)=\int_\G e^{-i\<\log y\mid\xi\>}{\sf K}_a\big(x,y^{-1}x\big)dy\,.
\end{equation}
On the other hand, a short computation shows that ${\sf Ber}_\o(f)$ is an integral operator in $L^2(\G)$ with kernel
\begin{equation*}\label{thekernel}
\kappa_\o(f):=\int_\G\int_{\g^\sharp}f(z,\zeta)\,\o_{z,\zeta}\!\otimes\overline{\o_{z,\zeta}}\,dzd\zeta\,.
\end{equation*}
Hence one will have ${\sf Ber}_\o(f)={\sf Op}\big[a_\o(f)\big]$ if and only if 
\begin{equation*}\label{iff}
\begin{aligned}
\big[a_\o(f)\big](x,\xi)&=\int_\G\int_\G\int_{\g^\sharp} e^{-i\<\log y\mid\xi\>}f(z,\zeta)\,\o_{z,\zeta}(x)\,\overline{\o_{z,\zeta}(y^{-1}x)}\,dy dz d\zeta\\
&=\int_\G\int_\G\int_{\g^\sharp} e^{-i\<\log y\mid\xi\>}f(z,\zeta)\,e^{-i\<\log(zx)\mid\zeta\>}\o(zx)\\
&\qquad\qquad\quad e^{i\<\log(zy^{-1}x)\mid\zeta\>}\,\overline{\o(zy^{-1}x)}\,dy dz d\zeta\,.
\end{aligned}
\end{equation*}
\end{proof}

\begin{Remark}\label{barra}
{\rm We can compute the pseudo-differential symbol of the (regularizing) operator $T$ in terms of the 
covariant symbol and the coherent states, using Proposition \ref{nucleuldansei} and formula \eqref{recup}. This means, at least formally, that $\,T={\sf Op}\big(a_T\big)$\,, with
\begin{equation*}\label{recupe}
\begin{aligned}
a_T(x,\xi)\!&=\!\int_\G e^{-i\<\log y\mid\xi\>}K_T\big(x,y^{-1}x\big)dy\\
&=\!\int_\G\int_\G\int_{\g^\sharp}\int_\G\int_{\g^\sharp}e^{-i\<\log y\mid\xi\>}\big[{\rm cov}_\o(T)\big](z,\zeta;z'\!,\zeta')\\
&\qquad\qquad\quad\qquad\quad\o_{z'\!,\zeta'}(x)\,\overline{\o_{z,\zeta}\big(y^{-1}x\big)}\,dy\,dzd\zeta\,dz'\!d\zeta'.
\end{aligned}
\end{equation*}}
\end{Remark}

\section{Other versions}\label{firroskos}

\subsection{{$\tau$}-quantizations}\label{firro}

Let $\tau:\G\to\G$ be any continuous map, that does not need to be a group morphism or to commute with 
inversion. The model is $x\to\tau x$ with $\tau\in[0,1]$ from the Abelian case $\G=\R^n$, but even in 
this simple case one can master much more than scalar transformations. In \cite{MR} such a parameter has
 been used in the global quantization involving the unitary dual $\wG$ of the group. In a final section, for nilpotent groups, it also appeared involved in generalizing 
the quantization \eqref{dulcisor} of symbols on $\Xi=\G\times\g^\sharp$, that may be replaced with
\begin{equation}\label{dulcisorel}
\big[{\sf Op}^\tau\!(a)u\big]\!(x)=\int_\G\int_{\g^\sharp}\!e^{i\<\log(y^{-1}x)\mid \xi\>}a\big(\tau(xy^{-1})^{-1}x,\xi\big)u(y)\,dyd\xi\,.
\end{equation}
Among other results, one can show the formula for the adjoint
\begin{equation*}\label{pidos}
{\sf Op}^\tau\!(a)^*={\sf Op}^{\tilde\tau}\!(\overline a)\,,\quad{\rm where}\quad\tilde\tau(x):=\tau(x^{-1})x\,.
\end{equation*}
Note that $\tau(\cdot)=\e$ corresponds to the identity map $\tilde\tau(x)=x$\,, switching from the left 
to the right quantization, and vice versa. (For the right quantization $a(y,\xi)$ appears in 
\eqref{dulcisorel}). Thus the Hilbert space adjoint corresponds to complex conjugation of symbols if and 
only if $\tau=\tilde\tau$. If $\G=\R^n$ (with addition) the number $\tau=1/2$ solves this and corresponds 
to the Weyl quantization.

\smallskip
In \cite[Sect.\,4]{MR} the existence problem of such a symmetric parameter $\tau$ has been tackled for 
very general groups. In particular, a natural solution has been found for our nilpotent case, based on 
the vector structure of the Lie algebra on the fact that the group and the Lie algebra are diffeomorphic.
 Explicitly, one sets
\begin{equation*}\label{simetraca}
\tau(x):=\int_0^1\exp[s\log x]ds\,.
\end{equation*}

Keeping $\tau$ arbitrary, we briefly (and formally) indicate in the sequel some of the modifications 
needed in the present paper to accommodate the quantization parameter $\tau$.

\smallskip
Instead of \eqref{friedar}, one can start with the family of unitary operators in $L^2(\G)$
\begin{equation*}\label{taulife}
\big[{\sf W}^\tau\!(z,\zeta)u\big](x):=e^{i\<\log[\tau(z)^{-1}x]\mid\zeta\>}u(z^{-1}x)\,,\quad(z,\zeta)\in\G\times\g^\sharp,
\end{equation*}
coinciding with those from \eqref{friedar} if $\tau(\cdot)\!:=\e$\,. We recall the formula 
${\sf W}(z,\zeta)\equiv{\sf W}^{\e}(z,\zeta)={\sf M}_\zeta{\sf L}_z$ (multiplications are placed to the left).
 For $\tau={\rm id}$ one has the opposite ordering  ${\sf W}^{\rm id}(z,\zeta)={\sf L}_z{\sf M}_\zeta$\,.

\smallskip
Then the $\tau$-Fourier-Wigner transform will be
\begin{equation*}\label{unica}
\begin{aligned}
\mathscr W^\tau_{u,v}(z,\zeta)\!:=\big\<{\sf W}^\tau\!(z,\zeta)u\!\mid\!v\big\>&=\int_\G e^{i\<\log[\tau(z)^{-1}y]\mid\zeta\>}u\big(z^{-1}y\big)\overline{v(y)}dy\\
&=\int_\G e^{i\<\log x\mid\zeta\>}u(z^{-1}\tau(z)x)\overline{v(\tau(z)x)}dx\,.
\end{aligned}
\end{equation*}
It consists of a partial Fourier transformation composed with a $\tau$-depending change of variable.

\smallskip
Computing the adjoint of ${\sf W}^\tau\!(z,\zeta)$ leads to coherent states built upon $\o\in\S(\G)$  and 
depending on $\tau$\,:
\begin{equation*}\label{noicoieri}
\o^\tau_{z,\zeta}(x):=\big[{\sf W}^\tau\!(z,\zeta)^*\o\big](x)=e^{-i\<\log[\tau(z)^{-1}zx]\mid\zeta\>}\o(zx)\,.
\end{equation*}
This releases a sequence of $\tau$-analogs of many of the notions and formulas above, with similar 
properties. For instance, a (slightly formal) expression for the $\tau$-Berezin quantization is
\begin{equation*}\label{tauberezin}
\begin{aligned}
&\big[{\sf Ber}^\tau_\o(f)u\big](x)\\
&=\int_\G\int_\G\int_{\g^\sharp}e^{i\<\log[\tau(z)^{-1}zy]-\log[\tau(z)^{-1}zx]\mid\zeta\>}f(z,\zeta)\,\overline{\o(zy)}\,\o(zx)\,u(y)\,dydzd\zeta\,.
\end{aligned}
\end{equation*}

The reader may formulate other results in the setting of the $\tau$-Berezin quantization. We only 
indicate another covariance result, valid for $\tau(x)=x$\,, that is different from 
Proposition \ref{crooks} (in a very non-commutative setting ordering issues do matter if one wants 
simple formulas).

\begin{Proposition}\label{croks}
For any $\zeta\in\g^\sharp$ and $f\in L^\infty(\G\times\g^\sharp)$\,, one has
\begin{equation*}\label{curvar}
{\sf M}_\zeta^*{\sf Ber}^{\rm id}_\o(f){\sf M}_\zeta={\sf Ber}^{\rm id}_\o\big[f(\cdot,\cdot-\zeta)\big]\,.
\end{equation*}
\end{Proposition}

\begin{proof}
Using notations from section \ref{firica}, one has
$$
\begin{aligned}
{\sf M}_\zeta^*\,{\sf Ber}^{\rm id}_\o(f){\sf M}_\zeta u&=\int_\G\int_{\g^\sharp}f(x,\xi)\,{\sf M}_\zeta^*\,\O^{\rm id}_{x,\xi}{\sf M}_\zeta u\,dxd\xi\\
&=\int_\G\int_{\g^\sharp}f(x,\xi)\,\big\<u,{\sf M}_\zeta^*{\sf W}^{\rm id}(x,\xi)^*\o\big\>{\sf M}_\zeta^*{\sf W}^{\rm id}(x,\xi)^*\o\,dxd\xi\\
&=\int_\G\int_{\g^\sharp}f(x,\xi)\,\big\<u,{\sf M}_\zeta^*{\sf M}^*_\xi{\sf L}^*_x\o\big\>{\sf M}_\zeta^*{\sf M}^*_\xi{\sf L}^*_x\o\,dxd\xi\\
&=\int_\G\int_{\g^\sharp}f(x,\xi)\,\big\<u,{\sf M}^*_{\zeta+\xi}{\sf L}^*_x\o\big\>{\sf M}_{\zeta+\xi}^*{\sf L}^*_x\o\,dxd\xi\\
&=\int_\G\int_{\g^\sharp}f(x,\xi)\,\big\<u,{\sf W}^{\rm id}(x,\zeta+\xi)^*\o\big\>{\sf W}^{\rm id}(x,\zeta+\xi)^*\o\,dxd\xi\\
&=\int_\G\int_{\g^\sharp}f(x,\xi)\,\O^{\rm id}_{x,\zeta+\xi}u\,dxd\xi\\
&=\int_\G\int_{\g^\sharp}f(x,\eta-\zeta)\,\O^{\rm id}_{x,\eta}(u)\,dxd\eta\,.
\end{aligned}
$$
\end{proof}

\subsection{Magnetic quantization}\label{firrak}

In the same setting of a connected simply connected nilpotent group $\G$\,, we consider 
{\it a magnetic field} $B$\,, i.\,e.\,a closed $2$-form on $\G$\,. It can be written as $B=dA$ for 
some $1$-form ({\it vector potential}). Any other vector potential $\tilde A$ satisfying $B=d\tilde A$ 
is related to the first by $\tilde A=A+d\psi$\,, where $\psi$ is a smooth function on $\G$\,; it would
 lead to a unitarily equivalent formalism (gauge covariance).

\smallskip
For $x,y\in\G$ one defines the smooth function $[x,y]:\mathbb R\rightarrow\G$ by
\begin{equation*}\label{abb}
[x,y]_s:=\exp[(1-s)\log x+s\log y]=\exp[\log x+s(\log y-\log x)]\,.
\end{equation*}
Its range $[[x,y]]:=\big\{\,[x,y]_s\!\mid\!s\in[0,1]\,\big\}$ is the segment in $\G$ connecting 
$x$ to $y$\,. The circulation of the $1$-form $A$ through the segment $[[x,y]]$ is 
\begin{equation*}\label{circ}
\Gamma^A[[x,y]]\equiv\int_{[[x,y]]}\!A:=\int_0^1\!\big\<\log y-\log x\,\big\vert\, A\big([x,y]_s\big)\big\>\,ds\,.
\end{equation*}
This leads to the following magnetic modification of the quantization \eqref{dulcisor}
$$
[{\sf Op}^{A}(a)u](x)=\int_{\G}\int_{\mathfrak g^\sharp} e^{i\int_{[[x,y]]}\!A}\,e^{i\<\log(xy^{-1})\mid\xi\>}a\big(x,\xi\big)u(y)\,dyd\xi\,,
$$
that has been introduced in \cite[Sect.\,4]{BM}. One finds in \cite{BM} more general constructions, 
consisting in twisting by $2$-cocycles pseudo-differential formalisms attached to type I unimodular 
locally compact groups. The Abelian case $\G=\R^n$ is deeply studied in \cite{IMP1,IMP2,MPR}, mainly 
in connection with magnetic Schr\"odinger operators.

\smallskip
At a basic level, the key modification is to replace the left regular representation ${\sf L}:\G\to\mathbb B[L^2(\G)]$ with {\it the family of left magnetic translations}
\begin{equation}\label{leftmag}
\big[{\sf L}^A_z(u)\big](x):=e^{i\int_{[[x,z^{-1}x]]}\!A}\,u\big(z^{-1}x\big)\,.
\end{equation}
They do not even form a projective representation. By using Stokes' Theorem one checks that
\begin{equation}\label{magnetoform}
{\sf L}^A_y{\sf L}^A_z=\O^B(y,z){\sf L}^A_{yz}\,,\quad\forall\,y,z\in\G\,,
\end{equation}
with $\O^B(y,z)$ the operator of multiplication by the function $x\to e^{\Gamma^B(x;y,z)}$, 
where $\Gamma^B(x;y,z)$ is the flux of the magnetic field $B$ through the "triangle" in $\G$ with 
corners $x,y^{-1}x$ and $z^{-1}y^{-1}x$\,, defined by "segments" of the form $[\![a,b]\!]$ as defined 
above. So there is a magnetic contribution to the Canonical Commutation Relations. 

\smallskip
Consequently, one defines the family of unitary operators ${\sf W}^A(z,\zeta):={\sf M}_\zeta{\sf L}^A_z$ in $L^{2}(\G)$ ({\it the magnetic Weyl system}, labeled by $\Xi=\G\times\g^\sharp$) by
\begin{equation*}\label{vacio}
\big[{\sf W}^A(z,\zeta)u\big](x):=e^{i\<\log x\mid\zeta\>}e^{i\int_{[[x,z^{-1}x]]}\!A}\,u(z^{-1}x)\,.
\end{equation*}
This leads to {\it magnetic coherent states}
\begin{equation}\label{coermag}
\o^A_{z,\zeta}(x):=\big[{\sf W}^A(z,\zeta)^*\o\big](x)=e^{-i\<\log(zx)\mid \zeta\>}e^{-i\int_{[[zx,x]]}\!A}\,\o(zx)\,,
\end{equation}
and {\it the magnetic Fourier-Wigner transform}
\begin{equation}\label{viognier}
\mathscr W^A_{u,v}(z,\zeta):=\big\<{\sf W}^A(z,\zeta)u,v\big\>=\int_\G e^{i\<\log y\mid\zeta\>}e^{i\int_{[[y,z^{-1}y]]}\!A}\,u(z^{-1}y)\,\overline{v(y)}\,dy\,.
\end{equation}
The output is {\it a magnetic Berezin quantization}
\begin{equation*}\label{Aberezin}
\begin{aligned}
\big[{\sf Ber}^A_\o(f)u\big](x)&=\int_\G\int_\G\int_{\g^\sharp}e^{i\<\log(zy)-\log(zx)\mid\zeta\>}\exp\Big\{i\Big(\int_{[[zy,y]]}\!\!\!A\,-\int_{[[zx,x]]}\!\!\!A\Big)\Big\}\\
&\qquad\qquad\quad f(z,\zeta)\,\overline{\o(zy)}\,\o(zx)\,u(y)\,dydzd\zeta\,.
\end{aligned}
\end{equation*}

The reader can easily extend the results of the main body of this article to the magnetic case. 
The $\tau$-quantizations are also possible in this set up.



\begin{thebibliography}{00}

\bibitem{AAG} S. T. Ali, J-P. Antoine and J-P. Gazeau: \textit{Coherent States, Wavelets and Their Generalizations}, Springer-Verlag, New York, 2000.

\bibitem{BKG} H. Bahouri, C. Fermanian-Kammerer and I. Gallagher: \textit{Phase Space Analysis and Pseudodifferential Calculus on the Heisenberg Group}, Asterisque, {\bf 342}, (2012).

\bibitem{BM} H. Bustos and M. M\u antoiu: \textit{Twisted Pseudo-differential Operators on Type I Locally Compact Groups}, Illinois J. Math., \textbf{60}(2), 365--390, (2016).

\bibitem{BBR} P. Boggiatto, E. Buzano and L. Rodino: \textit{Global Hypoellipticity and Spectral Theory}, Akademie-Verlag, 1996.

\bibitem{BOW} P. Boggiatto, A. Oliaro and M.\.W. Wong: \textit{$L^p$ Boundedness and Compactness of Localization Operators}, J. Math. Anal. Appl. \textbf{322}, 193--206, (2006).

\bibitem{CR} M. Combescure and D. Robert: \textit{Coherent States and Applications in Mathematical Physics}, Springer Dordrecht Heidelberg London New York, 2012. 

\bibitem{CGr} E. Cordero and K. Gr\"ochenig: \textit{Time-Frequency Analysis of Gabor Localization Operators}, J. Funct. Anal. \textbf{205}(1), 107--131, (2003).

\bibitem{CT} E. Cordero and A. Tabacco: \textit{Localization Operators via Time-Frequency Analysis}, Operator Theory: Advances and Applications, \textbf{155}, 131--146, Birkh\"auser Verlag Basel, 2004.

\bibitem{CG} L. J. Corwin and F. P. Greenleaf: \textit{Representations of Nilpotent Lie Groups and Applications}, Cambridge Univ. Press, 1990.

\bibitem{FR1} V. Fischer and M.  Ruzhansky: \textit{Quantization on Nilpotent Lie Groups}, Progress in Mathematics, \textbf{314}, Birkh\"auser, 2016. 

\bibitem{Glo1}P. G\l owacki: \textit{A Symbolic Calculus and $L^2$-Boundedness on Nilpotent Lie Groups}, J. Funct. Anal. {\bf 206}, 233--251, (2004).

\bibitem{Glo2} P. G\l owacki: \textit{The Melin Calculus for General Homogeneous Groups,} Ark. Mat., {\bf 45}(1), 31--48, (2007).

\bibitem{Glo3} P. G\l owacki: \textit{Invertibility of Convolution Operators on Homogeneous Groups}, Rev. Mat. Iberoam. {\bf 28}(1), 141--156, (2012).

\bibitem{GMP} A. Grossmann, J. Morlet and T. Paul: \textit{Transforms Associated to Square Integrable Group Representations I: General Results}, J. Math. Phys. {\bf 26}, (1985).

\bibitem{Ha} B. C. Hall, \textit{Holomorphic Methods in Analysis and Mathematical Physics}, Contemp. Math. \textbf{260}, 1--59 (2000).

\bibitem{IMP1} V. Iftimie, M. M\u antoiu and R. Purice: \textit{Magnetic Pseudodifferential Operators}, Publ. RIMS. {\bf 43}, 585--623, (2007).

\bibitem{IMP2} V. Iftimie, M. M\u antoiu and R. Purice: \textit{Commutator Criteria for Magnetic Pseudodifferential Operators}, Commun. PDE {\bf 35}, 1058--1094, (2010).

\bibitem{Ma1} D. Manchon: \textit{Formule de Weyl pour les groupes de Lie nilpotentes}, J. Reine Angew. Mat. {\bf 418}, 77--129, (1991).

\bibitem{Ma2} D. Manchon: \textit{Calcul symbolyque sur les groupes de Lie nilpotentes et applications}, J. Funct. Anal. {\bf 102} (2), 206--251, (1991).

\bibitem{M} M. M\u antoiu: \textit{A Positive Quantization on Type $I$ Locally Compact Groups}, Mathematische Nachrichten, in press.

\bibitem{MPR} M. M\u antoiu, R. Purice and S. Richard, \textit{Spectral and Propagation Results for Magnetic Schr\"odinger Operators; a $C^*$-Algebraic Framework}, J. Funct. Anal. {\bf 250}, 42--67, (2007).

\bibitem{MR} M. M\u antoiu and M. Ruzhansky: \textit{Pseudo-differential Operators, Wigner Transform and Weyl Systems on Type I Locally Compact Groups}, Doc. Math., {\bf 22 }, 1539--1592, (2017).

\bibitem{MR1} M. M\u antoiu and M. Ruzhansky: \textit{Quantizations on Nilpotent Lie Groups and Algebras Having Flat Coadjoint Orbits}, J. Geometric Analysis.

\bibitem{Me} A. Melin: \textit{Parametrix Constructions for Right Invariant Differential Operators on Nilpotent Groups}, Ann. Global Anal. Geom. {\bf 1}(1), 79--130, (1983).

\bibitem{Mi} K. Miller: \textit{Invariant Pseudodifferential Operators on Two Step Nilpotent Lie Groups}, Michigan Math. J. {\bf 29}, 315--328, (1982).

\bibitem{Mi1} K. Miller: \textit{Inverses and Parametrices for Right-Invariant Pseudodifferential Operators on Two-Step Nilpotent Lie Groups}, Trans. of the AMS, {\bf 280}(2), 721--736 (1983).

\bibitem{RT} M. Ruzhansky and V. Turunen: \textit{Pseudodifferential Operators and Symmetries}, Pseudo-Differential Operators: Theory and Applications {\bf 2}, Birkh\"auser Verlag, 2010. 

\bibitem{TB} J. Toft and P. Boggiatto: \emph{Schatten Classes for Toeplitz Operators with Hilbert Space Windows on Modulation Spaces}, Advances in Mathematics, \textbf{217}, 305--333, (2008). 

\bibitem{Wo} M.W. Wong: \emph{Wavelet Transforms and Localization Operators}, Basel; Boston; Berlin: Birkh\"auser, 2002.

\bibitem{Wo1} M. W. Wong: \emph{$L^p$ Boundedness of Localization Operators Associated to Left Regular Representations}, Proc. of the AMS, \textbf{130}(10), 2911--2919, (2002).

\end{thebibliography}
\end{document}